\documentclass[11pt,oneside,onecolumn,reqno]{amsart}

\usepackage{amssymb,latexsym,mathrsfs,
amsxtra,amscd,amsfonts,amsthm,amsmath,verbatim,epsfig}

\usepackage{color}

\usepackage[colorlinks,pagebackref=true]{hyperref}

\usepackage[english]{babel}

\newcommand{\arxiv}[1]{\href{http://arxiv.org/abs/#1}{{\tt arXiv:#1}}}

\makeatletter

\makeatletter
\def\widebreve#1{\mathop{\vbox{\m@th\ialign{##\crcr\noalign{\kern3\p@}%
      \brevefill\crcr\noalign{\kern3\p@\nointerlineskip}%
      $\hfil\displaystyle{#1}\hfil$\crcr}}}\limits}

\def\brevefill{$\m@th \setbox\z@\hbox{$\braceld$}%
  \bracelu\leaders\vrule \@height\ht\z@ \@depth\z@\hfill\braceru$}
\makeatletter

\def\@citecolor{blue}
\def\@linkcolor{blue}
\def\@urlcolor{blue}
\def\@urlcolor{blue}

\numberwithin{equation}{section}

\def\dim{\operatorname{dim}}
\def\depth{\operatorname{depth}}

\def\grade{\operatorname{grade}}

\def\reg{\operatorname{reg}}

\def\Proj{\operatorname{Proj}}

\def\ZZ{\mathbb Z}
\def\QQ{\mathbb Q}

\newcommand{\NN}{\mathbb N}
\newcommand{\mm}{\mathfrak m}

\newcommand{\ov}{\overline}
\newcommand{\lm}{{\lambda}}

\newcommand{\R}{\mathcal R}

\newcommand{\PP}{\mathcal P}
\newcommand{\I}{{\bf I}}
\newcommand{\n}{{\bf{n}}}

\newcommand{\m}{{\bf {m}}}
\newcommand{\rrr}{{\bf{r}}}
\newcommand{\ii}{i=1,\ldots,s}
\newcommand{\idl}{{I_1},\ldots,{I_s}}
\newcommand{\id}{{I_1}\cdots{I_s}}

\newcommand{\lf}{\left(}
\newcommand{\rg}{\right)}
\newcommand{\fil}{\mathcal F}

\newcommand{\ga}{{G(\fil)}}

\newcommand{\po}{P_{\fil}}
\newcommand{\ho}{H_{\fil}}
\newcommand{\hf}{H_{\I}}
\newcommand{\pf}{P_{\I}}
\newcommand{\crr}{complete reduction }
\newcommand{\bl}{\begin{lemma}}
\newcommand{\el}{\end{lemma}}
\newcommand{\bt}{\begin{theorem}}
\newcommand{\et}{\end{theorem}}
\newcommand{\ben}{\begin{enumerate}}
\newcommand{\een}{\end{enumerate}}
\newcommand{\bpf}{\begin{proof}}
\newcommand{\eepf}{\end{proof}}
\newcommand{\beqn}{\begin{eqnarray*}}
\newcommand{\eeqn}{\end{eqnarray*}}
\newcommand{\beqnn}{\begin{eqnarray}}
\newcommand{\eeqnn}{\end{eqnarray}}
\newcommand{\bd}{\begin{definition}}
\newcommand{\ed}{\end{definition}}
\newcommand{\bp}{\begin{proposition}}
\newcommand{\ep}{\end{proposition}}
\newcommand{\bc}{\begin{corollary}}
\newcommand{\ec}{\end{corollary}}
\newcommand{\bex}{\begin{example}}
\newcommand{\eex}{\end{example}}

\newcommand{\wrt}{with respect to }
\newcommand{\CM}{Cohen-Macaulay }
\newcommand{\wlg}{ Without loss of generality }

\theoremstyle{plain}
\newtheorem{theorem}{Theorem}[section]
\newtheorem{corollary}[theorem]{Corollary}
\newtheorem{proposition}[theorem]{Proposition}
\newtheorem{lemma}[theorem]{Lemma}

\newtheorem{example}[theorem]{Example}
\newtheorem{definition}[theorem]{Definition}

\theoremstyle{remark}
\newtheorem{remark}[theorem]{Remark}
\numberwithin{equation}{theorem}
\begin{document}
\title[multigraded regularity, reduction vectors and postulation vectors]{multigraded regularity, reduction vectors and postulation vectors}
\author{Parangama Sarkar}
\address{ Mathematics Department, Indian Institute of Technology Bombay, Mumbai, India 400076}
\email{parangama@math.iitb.ac.in}

\thanks{The author is supported by CSIR Fellowship of Government of India.} 

\begin{abstract}
We relate the set of complete reduction vectors  of a $\ZZ^s$-graded admissible filtration of ideals $\fil$ with the set of multigraded regularities of $G(\fil).$ We prove $\reg (G(\fil))=\reg(\R(\fil)).$ We establish a relation between the sets of complete reduction vectors of $\fil$ and postulation vectors of $\fil$ under some cohomological conditions.
\end{abstract}
\maketitle

\thispagestyle{empty}
\section{introduction}
Let $R=\bigoplus\limits_{n\in\NN}R_{n}$ be a standard Noetherian $\NN$-graded ring defined over a local ring $(R_{{{0}}},\mm)$ and $M=\bigoplus\limits_{n\in\ZZ}M_{n}$ be a finitely generated $\ZZ$-graded $R$-module. Let \[ a(M) = \left\{
  \begin{array}{l l}
   {\max\{n\mbox{ }|\mbox{ } M_n\neq 0\}}  & \quad \text{if $M\neq 0$ }\\
    {-\infty} & \quad \text{if $M=0.$ }
  \end{array} \right. \] The Castelnuovo-Mumford regularity of $M$ is defined by $$\reg(M)=\max\{a(H_{R_{+}}^i(M))+i\mbox{ }|\mbox{ }i\geq 0\}.$$ The Castelnuovo-Mumford regularity is a fundamental invariant of a module or a sheaf which measures 
its complexity. A. Ooishi \cite[Lemma 4.8]{aooishi}, M. Johnson and B. Ulrich \cite[Proposition 4.1]{JU} proved that $\reg(\R(I))=\reg(G(I))$ for any ideal $I$ in a Noetherian local ring $(R,\mm).$ Later, observing the close relationships of the invariants $a(H_{\R(I)_{+}}^i(\R(I)))$ and $a(H_{G(I)_{+}}^i(G(I))),$ Trung \cite[Corollary 3.3]{trung}, proved the same result. He also showed that for an $\mm$-primary ideal $I$ in a Noetherian local ring of dimension $d\geq 1$ and a  minimal reduction $J$ of $I,$ $$a_d+d\leq r_J(I)\leq\{a_i+i:i=0,\ldots,d\}$$ where $a_i=a(H_{G(I)_{+}}^i(G(I)))$ \cite[Proposition 3.2]{T1}. Using this relation, in $1993,$ Marley \cite[Corollary 2.2]{marley2}, proved that for an $\mm$-primary ideal $I$ in a \CM local ring of dimension $d\geq 1,$ if $\depth G(I)\geq d-1$ then $r(I)=n(I)+d=a_d+d$ where $n(I)$ is the postulation number of $I$ and $r(I)$ is the reduction number of $I. $
\\In this context, it is natural to ask : What is an analogue of Castelnuovo-Mumford regularity in multigraded case and relation between multigraded regularity of multi-Rees algebra and associated multigraded rings. 
\\The objective of this paper is to study multigraded regularity of blowup algebras, reduction vectors and postulation vectors of a $\ZZ^s$-graded admissible filtration of ideals. Among the themes presented are : $(1)$ equality of the sets of multigraded regularities of two blowup algebras associated to a $\ZZ^s$-graded admissible filtration of ideals $\fil,$ $(2)$ relation between the set of complete reduction vectors of $\fil$ and the set of multigraded regularities of blowup algebras associated to it and $(3)$ relation among complete reduction vectors and postulation vectors of $\fil$ under some cohomological conditions.
\\First we set up notation and recall few definitions. Throughout this paper, $(R,\mm)$ denotes a Noetherian local ring of dimension $d$ with infinite residue field. Let $ I_1,\ldots,I_s$ be $\mm$-primary ideals of $R.$ We denote the collection $(\idl)$ by $\I.$ For $s\geq 1,$ we  put ${\bf {e}}=(1,\ldots,1),\; {\bf{0}}=
 (0,\ldots,0)$ and ${\bf {e_i}}=(0,\ldots,1,\ldots,0)\in{\ZZ}^s$ for all $\ii$ where $1$ occurs at $i$th position. For $\n=(n_1,\ldots,n_s)\in{\ZZ}^s,$ we write ${\I}^{\n}=I_{1}^{n_1}\cdots I_{s}^{n_s}$ and $\n^+=(n_1^+,\ldots,n_s^+)$ where $n_i^+=\max\{0,n_i\}.$ For $\alpha=({\alpha}_1,\ldots,{\alpha}_s)\in{\NN}^s,$ we put $|\alpha|={\alpha}_1+\cdots+{\alpha}_s.$ We define $\m\geq\n$ if $m_i\geq n_i$ for all $\ii.$ By the phrase ``for all large $\n$", we mean $\n\in\NN^s$ and $n_i\gg 0$ for all $\ii.$
 A set of ideals $\fil=\lbrace\fil(\n)\rbrace_{\n\in \ZZ^s}$ is called a $\ZZ^s$-graded {\it{$\I$-filtration}} if for all $\m,\n\in\ZZ^s,$
{\rm (i)} ${\I}^{\n}\subseteq\fil(\n),$
 {\rm (ii)} $\fil(\n)\fil(\m)\subseteq\fil(\n+\m)$  and {\rm (iii)} if $\m\geq\n,$ $\fil(\m)\subseteq\fil(\n).$ 

Let $t_1,\ldots,t_s$ be indeterminates. For $\n\in\ZZ^s,$ we put ${\bf t}^{\n}=t_{1}^{n_1}\cdots t_{s}^{n_s}$ and denote the $\NN^s$-graded {\it{Rees ring of $\fil$}} by $\mathcal{R}(\fil)=\bigoplus\limits_{\n\in \NN^s}
{\fil}(\n){\bf{t}}^{\n}$ and the $\ZZ^s$-graded {\it{extended Rees ring of $\fil$}} by 
$\mathcal{R}'(\fil)=\bigoplus \limits_ {\n\in{\ZZ}^s}{\fil}(\n){\bf{t}}^{\n}.$ For an $\NN^s$-graded ring $S=\bigoplus\limits_{\n\geq{\bf{0}}}S_{\n},$ we denote the ideal $\bigoplus \limits_{\n\geq {\bf{e}}}S_{\n}$ by $S_{++}.$ Let $\ga=\bigoplus\limits_{\n\in{\NN}^s}{{\fil}(\n)}/{{\fil}(\n+{\bf e})}$ be the {\it{associated multigraded ring of $\fil$ with respect to $\fil({\bf e}).$}} For $\fil=\{\I^{\n}\}_{\n \in \ZZ^s}$, we set $\mathcal R(\fil)=\mathcal R(\I),$ $\mathcal R^\prime(\fil)=
\mathcal R^\prime(\I),$ $\ga=G(\I)$ and $\mathcal R(\I)_{++}=\R_{++}.$  A $\ZZ^s$-graded $\I$-filtration $\fil=\lbrace\fil(\n)\rbrace_{\n\in \ZZ^s}$ of ideals in $R$ is 
called an $\I$-{\it{admissible filtration}} if 
${{\fil}(\n)}={\fil}(\n^+)$ for all $\n\in\ZZ^s$ and $\mathcal{R}'(\fil)$
 is a finite $\mathcal{R}'(\I)$
 -module.
For an $\mm$-primary ideal $I$ in a Noetherian local ring $(R,\mm)$ of dimension $d,$ P. Samuel \cite{P1} showed that for all large $n,$ the {\it Hilbert-Samuel function of $I,$} $\displaystyle H_I(n)=\lm\lf{R}/{I^n}\rg$ coincides with a polynomial $$\displaystyle P_I(n)=e_0(I)\binom{n+d-1}{d}-e_1(I)
\binom{n+d-2}{d-1}+\cdots+(-1)^d e_d(I)$$ of degree $d,$ called the {\it Hilbert-Samuel polynomial of $I.$} The coefficients $e_i(I)$ for $i=0,1,\ldots,d$ are integers, known as Hilbert coefficients of $I.$ The coefficient $e(I):=e_0(I)$ \index{multiplicity of $I$! $e(I)$} is called the {\it multiplicity of $I.$} Let $s\geq 2,$ $\idl$ be $\mm$-primary ideals and $\fil$ be a $\ZZ^s$-graded $\I$-admissible filtrations of ideals in a Noetherian local ring $(R,\mm)$ of dimension $d.$ For the {\it Hilbert function of $\fil,$} $\ho(\n)=\lm\lf{R}/{{\fil}(\n)}\rg,$ D. Rees \cite{rees3} proved that there exists a polynomial of total degree $d$, called the {\it Hilbert polynomial of $\fil,$} of the form
$$\po(\n)=\displaystyle\sum\limits_{\substack{\alpha=({\alpha}_1,\ldots,{\alpha}_s)\in{\NN}^s \\ 
|\alpha|\leq d}}(-1)^{d-|{\alpha}|}{e_{\alpha}}(\fil)\binom{{n_1}+{{\alpha}_1}-1}{{\alpha}_1}
\cdots\binom{{n_s}+{{\alpha}_s}-1}{{\alpha}_s}$$ such that 
$\po(\n)=\ho(\n)$ for all large $\n.$ Here $e_{\alpha}(\fil)$ are integers called the Hilbert coefficients of $\fil.$ Rees \cite{rees3} showed that ${e_{\alpha}}(\fil)>0$ for $|\alpha|=d.$ These are called the {\it mixed multiplicities of $\fil.$} B. Teissier \cite{T} and P. B. Bhattacharya \cite{B} showed existence of this polynomial for the filtrations $\{{\I}^{\n}\}_{\n\in\ZZ^s}$ and $\{{\I}^{\n}\}_{\n\in\ZZ^2}$ respectively. 
\\Let $I$ be an $\mm$-primary ideal in a Noetherian local ring $(R,\mm).$ An integer $n(\mathcal I)$ is called the {\it postulation number of a $\ZZ$-graded $I$-admissible filtration $\mathcal I=\lbrace I_n\rbrace_{n\in\ZZ}$} if $P_{\mathcal I}(n)=H_{\mathcal I}(n)$ for all $n> n(\mathcal I)$ and $P_{\mathcal I}(n(\mathcal I))\neq H_{\mathcal I}(n(\mathcal I)).$  Let $\fil=\lbrace\fil(\n)\rbrace_{\n\in \ZZ^s}$ be an $\I$-admissible filtration of ideals in a Noetherian local ring $(R,\mm)$ and $s\geq 2.$ A vector $\n\in{\ZZ}^s$ is called a {\it postulation vector of $\fil$}\index{postulation! vector of $\fil$} if $\ho(\m)=\po(\m)$ for all $\m\geq\n.$ We use the notation $\mathcal{P}(\fil)$ to denote the set of all postulation vectors of $\fil.$
\\ D. Maclagan and G. Smith \cite{MS} have developed a multigraded variant of Castelnuovo-Mumford regularity in terms of vanishing of mutigraded components of local cohomology modules. Let $R=\bigoplus_{\n\in\NN^s}R_{\n}$ be a standard Noetherian $\NN^s$-graded ring defined over a local ring $(R_{\bf 0},\mm).$ Let $M=\bigoplus_{\n\in\ZZ^s}M_{\n}$ be a finitely generated $\ZZ^s$-graded $R$-module. Let $\m\in\ZZ^s.$ We say {\it{$M$ is $\m$-regular}} if for all $\n\in\NN^s,$
\ben{
\item $\left[H_{R_{++}}^0(M)\right]_{\n+\m+{\bf e_j}}=0$ for all $j=1,\dots,s,$

\item $\left[H_{R_{++}}^i(M)\right]_{\n+\m-{\bf p}}=0$ for all $i\geq 1$ and ${\bf p}\in\NN^s$ such that $|{\bf p}|=i-1.$
}\een
We set $\reg(M):=\{\m\in\ZZ^s\mid M\mbox{ is }\m\mbox{-regular}\}.$ Note that if $s=1$ then the minimum of $\reg(M)$ coincides with Castelnuovo-Mumford regularity of $M.$
 For $k\in\NN,$ we say {\it{$M$ is $\m$-regular from level $k$}} if $M$ is $\m$-regular when $k=0$ otherwise $[H_{R_{++}}^i(M)]_{\n+\m-{\bf p}}=0\mbox{ for all }i\geq k\mbox{ and }\n,{\bf p}\in\NN^s\mbox{ such that }|{\bf p}|=i-1.$ We set $\reg^k(M):=\{\m\in\ZZ^s\mid M\mbox{ is }\m\mbox{-regular from level }k \}.$
\\ A {\it reduction} of a $\ZZ$-graded $I$-admissible filtration $\mathcal I$ is an ideal $J\subseteq I_1$ such that $JI_{n} = I_{n+1}$ for all large $n.$ A {\it minimal reduction} of $\mathcal I$ is a reduction of ${\mathcal I}$ minimal with respect to inclusion. For a  minimal reduction $J$ of $\mathcal I,$ the integer ${r_J}(\mathcal I)=\min\lbrace m\in\NN\mid JI_n=I_{n+1}\mbox{ for }n\geq m\rbrace$ is called the {\it reduction number of $\mathcal I$ \wrt $J$} \index{reduction number! of $\mathcal I$ \wrt $J$} and the integer ${r}(\mathcal I)=\min\lbrace {r_J}(\mathcal I):J\mbox{ is a minimal reduction of }\mathcal I\rbrace$ is called the {\it reduction number of $\mathcal I.$} For a sequence of $\mm$-primary ideals $\I,$ Rees \cite{rees3} introduced the concept of complete reductions and joint reductions to study the multigraded Hilbert function $\hf(\n)=\lm(R/\I^{\n}).$ D. Kirby and Rees \cite{Kirby-Rees} generalized it further in the setting of multigraded rings and modules.
\\Let $\fil=\lbrace\fil(\n)\rbrace_{\n\in \ZZ^s}$ be a $\ZZ^s$-graded $\I$-admissible filtration of ideals in $(R,\mm)$ and $d\geq 1.$ A set of elements $\mathcal A_{\fil}=\lbrace x_{ij}\in I_i\mid j=1,\ldots,d;\ii\rbrace$ is called a {\it complete reduction of $\fil$} if for all $\n \geq \m$ for some $\m\in\NN^s,$ $(y_1,\ldots,y_d)\fil(\n)=\fil(\n+{\bf e})$ where $y_j=x_{1j}\cdots x_{sj}$ for all $j=1,\ldots,d.$ A vector $\rrr\in\NN^s$ is called a {\it complete reduction vector of $\fil$ with respect to $\mathcal{A}_{\fil}$} if for all $\n\geq\rrr,$ $(y_1,\ldots,y_d){\fil(\n)}=\fil(\n+{\bf e}).$ We use the notation $\R_{\mathcal A_{\fil}}(\fil)$ to denote the set of all complete reduction vectors of $\fil$ \wrt the complete reduction $\mathcal A_{\fil}.$
\\Let ${\bf q}=(q_1,\dots,q_s)\in \NN^s$  and $|{\bf q}|=d.$ A set of elements $\mathcal A_{{\bf q}}(\fil)=\lbrace a_{ij}\in I_{i}\mid j=1,\ldots,q_i;\ii\rbrace$ is called a {\it joint reduction of $\fil$ of type ${\bf q}$} if for all $\n \geq \m$ for some $\m\in\NN^s,$ $\sum\limits_{i=1}^{s}{\sum\limits_{j=1}^{q_i}{a_{ij}\fil(\n-{\bf e_i})}}=\fil(\n).$ A vector $\m$ is called a {\it joint reduction vector of $\fil$ with respect to the
joint reduction $\mathcal A_{{\bf q}}(\fil)$} if the above equality holds for all $\n\geq\m.$
\\In order to detect reduction vectors of multigraded admissible filtrations, we  use the theory of filter-regular sequences for multigraded modules. Let $M=\bigoplus_{\n\in\ZZ^s}M_{\n}$ be a finitely generated $\ZZ^s$-graded $R$-module where $R=\bigoplus_{\n\in\NN^s}R_{\n}$ is a standard Noetherian $\NN^s$-graded ring defined over a local ring $(R_{\bf 0},\mm).$ A homogeneous element $a\in R$ is called an {\it{$M$-filter-regular}} if $(0:_{M}a)_{\n}=0$ for all large $\n.$
Let $a_1,\dots,a_r\in R$ be homogeneous elements. Then $a_1,\dots,a_r$ is called an {\it{$M$-filter-regular sequence}} if $a_i$ is $M/(a_1,\dots,a_{i-1})M$-filter-regular for all $i=1,\dots,r.$
\\ The paper is organized in the following way. In section 2, we discuss some preliminary results required in the following section. In section 3, we show that $\reg(G(\fil))\subseteq\NN^s$ in which the following result plays a crucial role.
\bt
Let $(R,\mm)$ be a Noetherian local ring of dimension $d\geq 1$ with infinite residue field and $\idl$ be $\mm$-primary ideals in $R.$ Let $\mathcal{F}=\lbrace\fil(\n)\rbrace_{\n\in{\ZZ}^s}$ be an $\I$-admissible filtration of ideals in $R$ and $\m\in\reg(G(\fil)).$ Then for any ${\bf q}\in\NN^s$ with $|{\bf q}|=d$ and $A=\{i\mid q_i\geq 1\},$ there exists a joint reduction ${\mathcal A_{\bf q}}(\fil)=\lbrace a_{ij}\in I_i:j=1,\ldots,q_i;\ii\rbrace$ of $\fil$ of type ${\bf q}$ such that $\m+\sum\limits_{i\in A}{\bf e_i}$ is a joint reduction vector of $\fil$ \wrt ${\mathcal A_{\bf q}}(\fil),$  i.e. for all $\n\geq \m+\sum\limits_{i\in A}{\bf e_i},$ $$\sum\limits_{i=1}^{s}{\sum\limits_{j=1}^{q_i}{a_{ij}\fil(\n-{\bf e_i})}}=\fil(\n).$$ 
\et
We generalize the equality $\reg(\R(I))=\reg(G(I))$ for $\ZZ^s$-graded admissible filtrations $\fil.$  
\bt
Let $(R,\mm)$ be a Noetherian local ring of dimension $d\geq 1$ with infinite residue field and $\idl$ be $\mm$-primary ideals in $R.$ Let $\mathcal{F}=\lbrace\fil(\n)\rbrace_{\n\in{\ZZ}^s}$ be an $\I$-admissible filtration of ideals in $R.$ Then $\reg(\R(\fil))=\reg(G(\fil)).$
\et
We show that in a Noetherian local ring of dimension $d\geq 1,$ for any complete reduction $\mathcal A_{\fil}=\lbrace a_{ij}\in I_i:j=1,\ldots,d;\ii\rbrace$  of $\fil$ with $y_j=a_{1j}\cdots a_{sj}$ for all $j=1,\ldots,d,$ there exist $f_1,\ldots,f_d\in G(\I)_{\bf e}$ such that $f_1,\ldots,f_d$ is an $G(\fil)$-filter-regular sequence and $(y_1^*,\dots,y_d^*)=(f_1,\dots,f_d)G(\I)$ where $y_j^*=y_j+{\I}^{2{\bf e}}\in G(\I)_{\bf e}$ for all $j=1,\ldots,d$ and using this we prove the following result.
\bt
Let $(R,\mm)$ be a Noetherian local ring of dimension $d\geq 1$ with infinite residue field and $\idl$ be $\mm$-primary ideals in $R$ with $\grade(\id)\geq 1.$ Let $\mathcal{F}=\lbrace\fil(\n)\rbrace_{\n\in{\ZZ}^s}$ be an $\I$-admissible filtration of ideals in $R$ and $\mathcal A_{\fil}=\lbrace a_{ij}\in I_i:j=1,\ldots,d;\ii\rbrace$ be any complete reduction of $\fil.$ Let $H_{\R(\I)_{++}}^i(\R(\fil))_{\n}=0$ for all $0\leq i\leq d-1$ and $\n\geq {\bf 0}.$ Then 
\ben
\item[(1)] for all $\rrr\in{\R_{\mathcal A_{\fil}}(\fil)}$ with $\rrr\geq (d-1){\bf e},$ we have $\rrr-(d-1){\bf e}\in \PP(\fil),$
\item[(2)] if $s\geq 2,$ for all $\n\in\PP(\fil),$ we have $\n+(d-1){\bf e}\in \R_{\mathcal A_{\fil}}(\fil).$
\een 	In particular, for $s\geq 2,$ there exists a one-to-one correspondence 
$$f:{\displaystyle{\Large{\Large{\PP(\fil)\longleftrightarrow\lbrace{\rrr\in{\R_{\mathcal A_{\fil}}(\fil)}\mid\rrr\geq (d-1){\bf e}}\rbrace}}}}$$ defined by $f(\n)=\n+(d-1){\bf e}$ where $f^{-1}(\rrr)=\rrr-(d-1){\bf e}.$
\et
 {\bf Acknowledgement:} The author would like to thank Prof. Jugal K. Verma for his encouragement and insightful conversations.
\section{preliminaries}
Let $R=\bigoplus\limits_{\n\in\NN^s}R_{\n}$ be a standard Noetherian $\NN^s$-graded ring defined over an Artinian local ring $(R_{{\bf{0}}},\mm).$ Let $M=\bigoplus\limits_{\n\in\NN^s}M_{\n}$ be a finitely generated $\NN^s$-graded $R$-module. Let ${\Proj}^s(R)$ denote the set of all homogeneous prime ideals $P$ in $R$ such that $R_{++}\nsubseteq P.$ By \cite[Theorem 4.1]{hhrz}, there exists a numerical polynomial $P_M\in \QQ[X_1,\dots,X_s]$ of total degree $d=\dim Supp_{++}(M)$ of the form $$P_M(\n)=\displaystyle\sum\limits_{\substack{\alpha=({\alpha}_1,\ldots,{\alpha}_s)\in{\NN}^s \\ 
|\alpha|\leq d}}(-1)^{d-|{\alpha}|}{e_{\alpha}}(M)\binom{{n_1}+{{\alpha}_1}-1}{{\alpha}_1}
\cdots\binom{{n_s}+{{\alpha}_s}-1}{{\alpha}_s}$$ such that $e_{\alpha}(M)\in\ZZ,$ $P_M(\n)=\lm_{R_{{\bf{0}}}}\lf M_{\n}\rg$ for all large $\n$ and $e_{\alpha}(M)\geq 0$ for all $\alpha\in\NN^s$ such that $|\alpha|=d.$
\\Let $R^{\Delta}=\bigoplus\limits_{n\in\NN}R_{n{\bf e}}$ and $M^{\Delta}=\bigoplus\limits_{n\in\NN}M_{n{\bf e}}.$ Then $M^{\Delta}$ is an $\NN$-graded $R^{\Delta}$-module. Since for all large $\n,$ $$P_M(n{\bf e})=\lambda_{R_{\bf 0}}(M_{n{\bf e}})=\lambda_{R_{\bf 0}}((M^{\Delta})_{n}),$$ there exists a polynomial $P_{M^{\Delta}}(X)\in \QQ[X]$ such that $P_{M^{\Delta}}(n)=P_{M}(n{\bf e})$ for all $n$ and $\deg P_{M^{\Delta}}(X)=$ total degree of $P_{M}(X_1,\dots,X_s).$
\begin{lemma} $\deg P_{M^{\Delta}}(X)\geq 0$ if and only if $H_{R_{++}}^0(M)\neq M.$\end{lemma}
\begin{proof}
Let $\{m_1,\ldots,m_r\}$ be a generating set of $M$ as $R$-module and $\deg{m_i}=d_i=(d_{i1},\ldots,d_{is})$ for all $i=1,\ldots,r.$ Let $a=\max\{|d_{ij}|: i=1,\ldots,r; j=1,\ldots,s\}.$ 
Then for all $\n\geq a{\bf e},$ \beqn R_{\n-a{\bf e}}M_{a{\bf e}}\subseteq M_{\n}&\subseteq& R_{\n-d_1}M_{d_1}+\cdots+R_{\n-d_r}M_{d_r}\\&=& R_{\n-a{\bf e}}R_{a{\bf e}-d_1}M_{d_1}+\cdots+R_{\n-a{\bf e}}R_{a{\bf e}-d_r}M_{d_r}\subseteq R_{\n-a{\bf e}}M_{a{\bf e}}.\eeqn
Let $\deg P_{M^{\Delta}}(X)\geq 0.$ Suppose $H_{R_{++}}^0(M)= M.$ Since $M$ is finitely generated, there exists an integer $k \geq 1$ such that $R_{++}^k M=0.$ Thus for all $\n\geq (a+k){\bf e},$ $M_{\n}=R_{\n-a{\bf e}}M_{a{\bf e}}=0.$ This contradicts that $\deg P_{M^{\Delta}}(X)\geq 0.$
\\Conversely, suppose $H_{R_{++}}^0(M)\neq M.$ If $\deg P_{M^{\Delta}}(X)=-1$ then $M_{n{\bf e}}=0$ for all large $n,$ say for all $n\geq k\geq 1.$ Let $m\in M_{\m}$ and $t\geq \max\{k,a\}.$ Then $$R_{++}^{t}m\subseteq M_{\m+t{\bf e}}=R_{\m+t{\bf e}-a{\bf e}}M_{a{\bf e}}\subseteq R_{\m}M_{t{\bf e}}=0$$ contradicts that $H_{R_{++}}^0(M)\neq M.$
\end{proof}
\begin{remark}
Let $(R,\mm)$ be a Noetherian local ring of dimension $d\geq 1$ with infinite residue field $k$ and $\idl$ be $\mm$-primary ideals in $R.$ Let $\mathcal{F}=\lbrace\fil(\n)\rbrace_{\n\in{\ZZ}^s}$ be an $\I$-admissible filtration of ideals in $R.$ Therefore by \cite[Proposition 2.5] {msv}, $G(\fil)$ is finitely generated $G(\I)$-module. Now for all large $\n,$ $$\lm\lf\frac{\fil(\n)}{\fil(\n+{\bf e})}\rg=\lm_{R}\lf\frac{R}{\fil(\n+{\bf e})}\rg-\lm_{R}\lf\frac{R}{\fil(\n)}\rg=\po(\n+{\bf e})-\po(\n)$$ is a numerical polynomial in $\QQ[X_1,\dots,X_s]$ of total degree $d-1$ for all large $\n$ and $e_{\beta}(G(\fil))> 0$ for all $\beta\in\NN^s$ where $|\beta|=d-1$ for $s\geq 2$ and $e_{0}(G(\fil))>0$ when $s=1.$ Since $\{\fil(n{\bf e})\}_{n\in\ZZ}$ is a $\ZZ$-graded $\{\I^{n{\bf e}}\}_{n\in\ZZ}$-admissible filtration, $\dim G(\fil)^{\Delta}=\deg P_{G(\fil)^{\Delta}}(X)+1=d.$
\end{remark}
\begin{lemma}\label{very very important}
Let $R=\bigoplus\limits_{\n\in\NN^s}R_{\n}$ be a standard Noetherian $\NN^s$-graded ring defined over an Artinian local ring $(R_{{\bf{0}}},\mm).$ Let $M=\bigoplus\limits_{\n\in\NN^s}M_{\n}$ be a finitely generated $\NN^s$-graded $R$-module and $e_{\alpha}(M)>0$ for all $\alpha\in\NN^s$ with $|\alpha|=\deg P_{M^{\Delta}}(X).$ Let $f\in R_{\bf e}$ be an $M$-filter-regular element. Then $e_{\beta}(M/fM)>0$ for all $\beta\in\NN^s$ with $|\beta|=\deg P_{(M/fM)^{\Delta}}(X)$ and $\deg P_{(M/fM)^{\Delta}}(X)=\deg P_{M^{\Delta}}(X)-1.$  
\end{lemma}
\bpf
Consider the exact sequence of $R$-modules $$0\longrightarrow (0:_{M}f )_{\n-{\bf {e}}}\longrightarrow M_{\n-{\bf {e}}}\overset{f}\longrightarrow M_{\n}\longrightarrow \frac{M_{\n}}{fM_{\n-{\bf {e}}}}\longrightarrow 0.$$ Since $f$ is $M$-filter-regular, for all large $\n,$ we get $$\lm_{R_{{\bf{0}}}}\lf\frac{M_{\n}}{fM_{\n-{\bf {e}}}}\rg=\lm_{R_{{\bf{0}}}}\lf M_{\n}\rg-\lm_{R_{{\bf{0}}}}\lf M_{\n-{\bf {e}}}\rg$$ and hence for all $\n\in\ZZ^s,$ $P_{{M}/{fM}}(\n)=P_M(\n)-P_M(\n-{\bf {e}}).$ Therefore for all $\beta\in\NN^s$ with $|\beta|=\deg P_{(M/fM)^{\Delta}}(X),$ $e_{\beta}(M/fM)=e_{\alpha_1}(M)+\ldots+e_{\alpha_s}(M)>0$ where ${\alpha_i}=\beta+{\bf e_i}$ for all $\ii.$
\eepf
\bl\label{re21}
Let $R=\bigoplus\limits_{\n\in\NN^s}R_{\n}$ be a $\NN^s$-graded Noetherian ring defined over a local ring $(R_{\bf 0},\mm)$ and $M=\bigoplus\limits_{\n\in\NN^s}M_{\n}$ be a finitely generated $\NN^s$-graded $R$-module. Let $a\in R_{\m}$ be an $M$-filter-regular where $\m\neq{\bf 0}.$ Then for all $\n\in\ZZ^s$ and $i\geq 0,$ the following sequence is exact
\beqn \displaystyle [H^{i}_{R_{++}}(M)]_{\n}\longrightarrow \left[H^{i}_{R_{++}}\lf\frac{M}{(aM)}\rg\right]_{\n}\longrightarrow [H^{i+1}_{R_{++}}(M)]_{\n-\m}.\eeqn 
\el
\bpf
Consider the following short exact sequence of $R$-modules $$0\longrightarrow (0:_{M}a)\longrightarrow M\longrightarrow \frac{M}{(0:_{M}a)}\longrightarrow 0.$$ Since $a$ is $M$-filter-regular, $(0:_{M}a)$ is $R_{++}$-torsion. Hence $$H^{i}_{R_{++}}(M)\simeq H^{i}_{R_{++}}\lf\frac{M}{(0:_{M}a)}\rg\mbox{ for all } i\geq 1.$$ Therefore the short exact sequence of $R$-modules $$0\longrightarrow \frac{M}{(0:_{M}a)}(-\m)\overset{a}\longrightarrow M\longrightarrow \frac{M}{aM}\longrightarrow 0$$ gives the desired exact sequence. 
\eepf
\bl\label{need}
Let $R=\bigoplus\limits_{\n\in\NN^s}R_{\n}$ be a $\NN^s$-graded Noetherian ring defined over a local ring $(R_{{\bf 0}},\mm)$ and $M=\bigoplus\limits_{\n\in\NN^s}M_{\n}$ be a finitely generated $\NN^s$-graded $R$-module. Let $H^{i}_{R_{++}}(M)_{\n}=0$ for all $i\geq 0$ and $\n\geq\m.$ Let $a_1,\dots,a_t\in R_{{\bf {e}}}$ be an $M$-filter-regular sequence. Then $H_{R_{++}}^i({M}/{(a_1,\dots,a_t)M})_{\n}=0$ for all $i\geq 0$ and $\n\geq \m+t{\bf e}.$
\el
\bpf
We use induction on $t.$ Let $t=1.$ By Lemma \ref{re21}, for all $i\geq 0,$ we get the exact sequence
 \beqn\label{re1}[H^{i}_{R_{++}}(M)]_{\n}\longrightarrow \left[H^{i}_{R_{++}}\lf\frac{M}{a_1M}\rg\right]_{\n}\longrightarrow [H^{i+1}_{R_{++}}(M)]_{\n-{\bf {e}}}.\eeqn Since for all $i\geq 0$ and $\n\geq \m,$ $[H^{i}_{R_{++}}(R)]_{\n}=0,$ we get $\displaystyle\left[H^{i}_{R_{++}}\lf\frac{M}{a_1M}\rg\right]_{\n}=0$ for all $i\geq 0$ and $\n\geq \m+{\bf {e}}.$ Hence the result is true for $l=1.$
\\ Suppose $t\geq 2$ and the result is true upto $t-1.$ Let $M'=M/(a_1,\dots,a_{l-1})M.$ Then $$\displaystyle [H^{i}_{R_{++}}(M')]_{\n}=0\mbox{ for all }i\geq 0\mbox{ and }\n\geq \m+(t-1){\bf {e}}.$$ Since $a_t$ is $M'$-filter-regular, using $t=1$ case, we get $\displaystyle \left[H^{i}_{R_{++}}\lf\frac{M'}{a_tM'}\rg\right]_{\n}=0$ for all $i\geq 0$ and $\n\geq \m+t{\bf {e}}.$   
\eepf
\bl\label{all}
Let $(R,\mm)$ be a Noetherian local ring of dimension $d\geq 1$ with infinite residue field and $\idl$ be $\mm$-primary ideals in $R.$ Let $\mathcal{F}=\lbrace\fil(\n)\rbrace_{\n\in{\ZZ}^s}$ be an $\I$-admissible filtration of ideals in $R.$ Let $y_1,\ldots,y_t\in\I^{\bf e}$ and $\m\in\NN^s$ such that $\fil(\n)=(y_1,\ldots,y_t)\fil(\n-{\bf e})+\fil(\n+{\bf e})$ for all $\n\geq\m.$ Then $\fil(\n)=(y_1,\ldots,y_t)\fil(\n-{\bf e})$ for all $\n\geq\m.$
\el
\bpf
Since $\fil$ is an $\I$-admissible filtration, by \cite[Proposition 2.4]{msv}, for each $\ii,$ there exists an integer $r_i$ such that for all $\n\in\ZZ^s,$ where $n_i\geq r_i,$ $\fil(\n+{\bf {e_i}})=I_i\fil(\n).$ Let $r=\max\{r_1,\ldots,r_s\}.$ For all $\n\geq \m,$ we get \beqn \fil(\n)&=&(y_1,\ldots,y_t)\fil(\n-{\bf e})+\fil(\n+{\bf e})\\&\subseteq& (y_1,\ldots,y_t)\fil(\n-{\bf e})+(y_1,\ldots,y_t)\fil(\n)+\fil(\n+2{\bf e})\\&\subseteq& \vdots\\&\subseteq& (y_1,\ldots,y_t)\fil(\n-{\bf e})+\ldots+(y_1,\ldots,y_t)\fil(\n+r{\bf e})+\fil(\n+(r+1){\bf e})\\&\subseteq& (y_1,\ldots,y_t)\fil(\n-{\bf e})+\I^{\bf e}\fil(\n+r{\bf e})\subseteq (y_1,\ldots,y_t)\fil(\n-{\bf e})+\mm\fil(\n).\eeqn Thus by Nakayama's Lemma, we get the required result.
\eepf
\bp\label{reghelp}
Let $R=\bigoplus\limits_{\n\in\NN^s}R_{\n}$ be a standard $\NN^s$-graded Noetherian ring defined over a local ring $(R_{{\bf 0}},\mm)$ and $M=\bigoplus\limits_{\n\in\ZZ^s}M_{\n}$ be a finitely generated $\ZZ^s$-graded $R$-module. Let $a_1,\ldots,a_l\in R_{e_j}$ is a $M$ filter-regular sequence. Suppose  for all $\n\in\NN^s,$ $\displaystyle\left[H_{R_{++}}^i(M)\right]_{\m+\n-{\bf p}}=0$ for all $i\geq 1$ and ${\bf p}\in\NN^s$ such that $|{\bf p}|=i-1.$ Then for all $\n\in\NN^s,$ 
$\displaystyle\left[H_{R_{++}}^i({M}/{(a_1,\ldots,a_l)M})\right]_{\m+\n-{\bf p}}=0$ for all $i\geq 1$ and ${\bf p}\in\NN^s$ such that $|{\bf p}|=i-1.$
\ep
\bpf
We prove it using induction on $l.$ Let $l=1.$ By Lemma \ref{re21}, for all $i\geq 1$ and $\n\in\ZZ^s,$ we have the following exact sequence  \beqn \displaystyle [H^{i}_{R_{++}}(M)]_{\n}\longrightarrow \left[H^{i}_{R_{++}}\lf\frac{M}{(a_1M)}\rg\right]_{\n}\longrightarrow [H^{i+1}_{R_{++}}(M)]_{\n-{\bf e_j}}.\eeqn Let $i\geq 1$ and ${\bf p}\in\NN^s$ such that $|{\bf p}|=i-1.$  Since $|{\bf e_j}+{\bf p}|=i,$ we have $\left[H_{R_{++}}^{i+1}(M)\right]_{\m+\n-{\bf e_j}-{\bf p}}=0$ for all $\n\in\NN^s.$ Thus we get $\displaystyle\left[H_{R_{++}}^i({M}/{a_1M})\right]_{\m+\n-{\bf p}}=0$ for all $i\geq 1$ and ${\bf p}\in\NN^s$ such that $|{\bf p}|=i-1.$ Let $l\geq 2,$ $N=M/a_1,\ldots,a_{l-1})M$ and $\displaystyle\left[H_{R_{++}}^i(N)\right]_{\m+\n-{\bf p}}=0$ for all $i\geq 1$ and ${\bf p}\in\NN^s$ such that $|{\bf p}|=i-1.$ Since $a_l$ is $N$-filter-regular element, the result follows from $l=1$ case.
\eepf
\section{multigraded regularity, reduction vectors and postulation vectors}
In this section we prove that for a $\ZZ^s$-graded $\I$-admissible filtrations of ideals, $$\reg(\R(\fil))=\reg(G(\fil))\subseteq \NN^s,$$ which generalizes a result of Trung \cite[Corollary 3.3]{trung}, Ooishi \cite[Lemma 4.8]{aooishi}, Johnson and Ulrich \cite[Proposition 4.1]{JU}. We show that if $\mathcal A_{\fil}=\lbrace a_{ij}\in I_i:j=1,\ldots,d;\ii\rbrace$ is  any complete reduction of a $\ZZ^s$-graded $\I$-admissible filtration of ideals in $R$ then there exist $f_1,\ldots,f_d\in \mathfrak a_{\bf e}$ such that $f_1,\ldots,f_d$ is an $G(\fil)$-filter-regular sequence and $(y_1^*,\dots,y_d^*)=(f_1,\dots,f_d)G(\I).$ We provide relations among complete reduction vectors and $\reg(G(\fil)).$ We conclude this section by establishing a relationship between reduction vectors \wrt complete reductions and postulation vectors of admissible multigraded filtrations of ideals in \CM local rings of dimension $d\geq 1$ under some cohomological conditions. This generalizes results of Sally \cite[Proposition 3]{sa}, Ooishi \cite[Proposition 4.10]{aooishi}, and Marley \cite[Corollary 2.2]{marley2}, \cite[Theorem 2]{marley}, \cite[Corollary 3.8]{marleythesis}.
\bt\label{important}
Let $(R,\mm)$ be a Noetherian local ring of dimension $d\geq 1$ with infinite residue field and $\idl$ be $\mm$-primary ideals in $R.$ Let $\mathcal{F}=\lbrace\fil(\n)\rbrace_{\n\in{\ZZ}^s}$ be an $\I$-admissible filtration of ideals in $R$ and $\m\in\reg(G(\fil)).$ Then for any ${\bf q}\in\NN^s$ with $|{\bf q}|=d$ and $A=\{i\mid q_i\geq 1\},$ there exists a joint reduction ${\mathcal A_{\bf q}}(\fil)=\lbrace a_{ij}\in I_i:j=1,\ldots,q_i;\ii\rbrace$ of $\fil$ of type ${\bf q}$ such that $\m+\sum\limits_{i\in A}{\bf e_i}$ is a joint reduction vector of $\fil$ \wrt ${\mathcal A_{\bf q}}(\fil),$  i.e. for all $\n\geq \m+\sum\limits_{i\in A}{\bf e_i},$ $$\sum\limits_{i=1}^{s}{\sum\limits_{j=1}^{q_i}{a_{ij}\fil(\n-{\bf e_i})}}=\fil(\n).$$ 
\et
\bpf
By \cite[Theorem 2.3]{powerproduct}, there exists a joint reduction ${\mathcal A_{\bf q}}(\fil)=\lbrace a_{ij}\in I_i:j=1,\ldots,q_i;\ii\rbrace$ of $\fil$ of type ${\bf q}$ such that 
$a^*_{11},\dots,a^*_{1q_1},\dots,a^*_{s1},\dots,a^*_{sq_s}$ is a $G(\fil)$-filter-regular sequence where $a^*_{ij}$ is the image of $a_{ij}$ in $G(\I)_{\bf {e_i}}$ for all $j=1,\ldots,q_i$ and $\ii.$
Since $\m\in\reg(G(\fil)),$ using Lemma \ref{re21} and Proposition \ref{reghelp} for $q_i$ times for each $i,$ we get $$\displaystyle \left[H^{0}_{G(\I)_{++}}\lf\frac{G(\fil)}{(a^*_{11},\dots,a^*_{1q_1},\dots,a^*_{s1},\dots,a^*_{sq_s})G(\fil)}\rg\right]_{\n}=0$$ for all $\n\geq \m+\sum\limits_{i\in A}{\bf e_i}.$ Now $\fil$ is an $\I$-admissible filtration. Hence by \cite[Proposition 2.4]{msv}, for each $\ii,$ there exists an integer $r_i$ such that for all $\n\in\ZZ^s,$ where $n_i\geq r_i,$ $\fil(\n+{\bf {e_i}})=I_i\fil(\n).$ Let $r=\max\{r_1,\ldots,r_s\}.$ Since ${G(\fil)}/{(a^*_{11},\dots,a^*_{1q_1},\dots,a^*_{s1},\dots,a^*_{sq_s})G(\fil)}$ is $G(\I)_{++}$-torsion, we get that for all $\n\geq \m+\sum\limits_{i\in A}{\bf e_i},$ \beqn
\fil(\n)&=&\sum\limits_{i=1}^{s}{\sum\limits_{j=1}^{q_i}{a_{ij}\fil(\n-{\bf e_i})}}+\fil(\n+{\bf e})\subseteq  \sum\limits_{i=1}^{s}{\sum\limits_{j=1}^{q_i}{a_{ij}\fil(\n-{\bf e_i})}}+\fil(\n+2{\bf e})\\&\subseteq & \vdots \\&\subseteq & \sum\limits_{i=1}^{s}{\sum\limits_{j=1}^{q_i}{a_{ij}\fil(\n-{\bf e_i})}}+\fil(\n+(r+1){\bf e})\subseteq  \sum\limits_{i=1}^{s}{\sum\limits_{j=1}^{q_i}{a_{ij}\fil(\n-{\bf e_i})}}+{\I}^{\bf e}\fil(\n).
\eeqn Therefore by Nakayama's Lemma, we get the required result.
\eepf
\bt\label{p4u}
Let $(R,\mm)$ be a Noetherian local ring of dimension $d\geq 1$ with infinite residue field and $\idl$ be $\mm$-primary ideals in $R.$ Let $\mathcal{F}=\lbrace\fil(\n)\rbrace_{\n\in{\ZZ}^s}$ be an $\I$-admissible filtration of ideals in $R.$ Then $\reg(G(\fil)\subseteq \NN^s.$
\et
\bpf
Suppose there exists $\m\in\ZZ^s\setminus\NN^s$ and $\m\in\reg(G(\fil)).$ Then there exists atleast one $i\in\{1,\ldots,s\}$ such that $m_i<0.$ \wlg assume $i=1.$ By Theorem \ref{important}, there exists a joint reduction $\lbrace a_{j}\in I_1:j=1,\ldots,d\rbrace$ of $\fil$ whose image in $G(\I)_{\bf e_1}$ is a $G(\fil)$-filter-regular sequence. Thus by Theorem \ref{important}, for all $\n\in\NN^s,$  we have $$\displaystyle{\left(\frac{G(\fil)}{(a_1^*,\ldots,a_d^*)G(\fil)}\right)_{\m+\n+{\bf e_1}}=0}.$$  Now $\fil$ is an $\I$-admissible filtration. Hence by \cite[Proposition 2.4]{msv}, for each $\ii,$ there exists an integer $r_i$ such that for all $\n\in\ZZ^s,$ where $n_i\geq r_i,$ $\fil(\n+{\bf {e_i}})=I_i\fil(\n).$ Let $r=\max\{r_1,\ldots,r_s\}.$ Define ${\bf k}=(k_1,\ldots,k_s)$ such that  
\[ k_j = \left\{
  \begin{array}{l l}
   {0}  & \quad \text{if $m_j\leq 0$ }\\
    {m_j} & \quad \text{$m_j\geq 0.$ }
  \end{array} \right. \] 
 Note that ${\bf k}\geq \m+{\bf e_1}.$ Thus for all $\n\in\NN^s,$ we have $\displaystyle\lf\frac{G(\fil)}{(a_1^*,\ldots,a_d^*)G(\fil)}\rg_{{\bf k}+\n}=0.$ Since $k_1= 0,$ we have $\fil({\bf k})=\fil({\bf k}+{\bf e})$ and 
 \beqn
 \fil({\bf k}+{\bf e})&= &{\sum_{j=1}^{d}{a_{j}\fil({\bf k}+{\bf e}-{\bf e_1})}}+\fil({\bf k}+2{\bf e})\subseteq  \sum_{j=1}^{d}{a_{j}\fil({\bf k}+{\bf e}-{\bf e_1})}+\sum_{j=1}^{d}{a_{j}\fil({\bf k}+2{\bf e}-{\bf e_1})}+\fil({\bf k}+3{\bf e})\\&\subseteq & \\&\vdots& \\&\subseteq & \sum_{j=1}^{d}{a_{j}\fil({\bf k}+{\bf e}-{\bf e_1})}+\fil({\bf k}+(r+1){\bf e})\subseteq\sum_{j=1}^{d}{a_{j}\fil({\bf k}+{\bf e}-{\bf e_1})}+{\I}^{\bf e}\fil({\bf k})\subseteq\mm\fil({\bf k}).
 \eeqn
Thus by Nakayama's Lemma, $\fil({\bf k})=0$ which is a contradiction. Hence $\reg(G(\fil))\subseteq\NN^s.$
\eepf
We recall the following result.
\bp\label{result 2}{\rm{\cite[Proposition 4.2]{msv}}}
Let $S^\prime$ be a $\ZZ^s$-graded ring and 
$S=\bigoplus\limits_{\n \in \NN^s} S^\prime_{\n}.$ 
Then 
$H^i_{S_{++}}(S^\prime) \cong H^i_{S_{++}}(S)$ for all $i >1 $ and the sequence 
$$0 \longrightarrow H^0_{S_{++}}(S) \longrightarrow H^0_{S_{++}}(S^\prime) \longrightarrow 
\frac{S^\prime}{S} \longrightarrow H^1_{S_{++}}(S) \longrightarrow H^1_{S_{++}}(S^\prime)
\longrightarrow 0 $$
is exact.
\ep
\bt\label{regpo}
Let $(R,\mm)$ be a Noetherian local ring of dimension $d\geq 1$ with infinite residue field and $\idl$ be $\mm$-primary ideals in $R.$ Let $\mathcal{F}=\lbrace\fil(\n)\rbrace_{\n\in{\ZZ}^s}$ be an $\I$-admissible filtration of ideals in $R.$ Then $\reg(\R(\fil))=\reg(G(\fil).$
\et
\bpf
First we prove that $\reg(G(\fil))\subseteq \reg(\R(\fil)).$ Consider the short exact sequence of $\R(\I)$-modules $$0\longrightarrow \mathcal{R}'(\fil)({\bf e})\longrightarrow \mathcal{R}'(\fil)\longrightarrow G'(\fil)\longrightarrow 0$$ which induces the long exact sequence of $R$-modules
$$\cdots\longrightarrow [H_{{\R}_{++}}^i(\mathcal{R}'(\fil))]_{\n+{\bf e}} \longrightarrow [H_{{\R}_{++}}^i(\mathcal{R}'(\fil))]_{\n} 
\longrightarrow [H_{{\R}_{++}}^i(G'(\fil))]_{\n}
 \longrightarrow  \cdots.$$ By Proposition \ref{result 2}, $H_{{\R}_{++}}^i(\mathcal{R}'(\fil))=H_{{\R}_{++}}^i(\mathcal{R}(\fil)),$ $H_{{\R}_{++}}^i(\mathcal{G}'(\fil))=H_{{\R}_{++}}^i(\mathcal{G}(\fil))$ for all $i\geq 2$ and $[H_{{\R}_{++}}^i(\mathcal{R}'(\fil))]_{\n}=[H_{{\R}_{++}}^i(\mathcal{R}(\fil))]_{\n},$ $[H_{{\R}_{++}}^i(\mathcal{G}'(\fil))]_{\n}=[H_{{\R}_{++}}^i(\mathcal{G}(\fil))]_{\n}$ for all $i=0,1$ and $\n\in\NN^s.$ Since by Theorem \ref{p4u}, $\reg(G(\fil))\subseteq\NN^s,$ we get $\reg(G(\fil))\subseteq \reg(\R(\fil)).$
 \\Conversely, suppose $\m\in\reg(\R(\fil)).$ Let $N$ be the $\R(\I)$-submodule of $\R(\fil)$ such that \[ N_{\bf k} = \left\{
  \begin{array}{l l}
   {0}  & \quad \text{if $k_j\leq 1$ for atleast one $j\in\{1,\ldots,s\}$ }\\
    \R(\fil)_{\bf k} & \quad \text{${\bf k}\geq{\bf e}.$ }
  \end{array} \right. \]
 Consider the short exact sequence of $\R(\I)$-modules $$0\longrightarrow N \longrightarrow \R(\fil)\longrightarrow C=\R(\fil)/N\longrightarrow 0$$ which induces a long exact sequence of modules $$\cdots\longrightarrow [H_{{\R}_{++}}^i(N)]_{\n} \longrightarrow [H_{{\R}_{++}}^i(\mathcal{R}(\fil))]_{\n} 
\longrightarrow [H_{{\R}_{++}}^i(C)]_{\n}
 \longrightarrow\cdots.$$ Since $C$ is $\R_{++}$-torsion, we have $H_{{\R}_{++}}^i(\mathcal{R}(\fil))=H_{{\R}_{++}}^i(N)$ for all $i\geq 2$ and \beqnn\label{tech} 0\longrightarrow [H_{{\R}_{++}}^0(N)] \longrightarrow [H_{{\R}_{++}}^0(\mathcal{R}(\fil))] 
\longrightarrow C \longrightarrow [H_{{\R}_{++}}^1(N)] \longrightarrow [H_{{\R}_{++}}^1(\mathcal{R}(\fil))] 
\longrightarrow 0 \eeqnn is exact. Now consider the short exact sequence of $\R(\I)$-modules $$0\longrightarrow N({\bf e}) \longrightarrow \R(\fil)\longrightarrow G(\fil)\longrightarrow 0$$ which induces a long exact sequence of modules $$\cdots\longrightarrow [H_{{\R}_{++}}^i(N)]_{\n+{\bf e}} \longrightarrow [H_{{\R}_{++}}^i(\mathcal{R}(\fil))]_{\n} 
\longrightarrow [H_{{\R}_{++}}^i(G(\fil))]_{\n}\longrightarrow  [H_{{\R}_{++}}^{i+1}(N)]_{\n+{\bf e}}
 \longrightarrow\cdots.$$ For all $i\geq 1,$ $$[H_{{\R}_{++}}^i(\mathcal{R}(\fil))]_{\n} 
\longrightarrow [H_{{\R}_{++}}^i(G(\fil))]_{\n}\longrightarrow  [H_{{\R}_{++}}^{i+1}(\R(\fil))]_{\n+{\bf e}}
$$ is exact. Therefore for all $i\geq 1,$ $[H_{{\R}_{++}}^i(G(\fil))]_{\m+\n-{\bf p}}=0$ for all $\n, {\bf p}\in\NN^s$ with $|{\bf p}|=i-1.$
\\Consider the exact sequence \beqnn\label{ki}[H_{{\R}_{++}}^0(\mathcal{R}(\fil))]_{\n} 
\longrightarrow [H_{{\R}_{++}}^0(G(\fil))]_{\n}\longrightarrow  [H_{{\R}_{++}}^{1}(N)]_{\n+{\bf e}}
 \longrightarrow [H_{{\R}_{++}}^1(\mathcal{R}(\fil))]_{\n}.\eeqnn Fix $j\in\{1,\ldots,s\}.$ Then $[H_{{\R}_{++}}^0(\mathcal{R}(\fil))]_{\m+\n+{\bf e_j}}=0=[H_{{\R}_{++}}^1(\mathcal{R}(\fil))]_{\m+\n}$ for all $\n\in\NN^s.$ Fix $\n\in\NN^s.$\\
 \\ {\bf Case 1:} Let $\m+\n+{\bf e_j}\in\ZZ^s\setminus\NN^s.$ Since $H_{{\R}_{++}}^0(G(\fil))\subseteq G(\fil),$ we have $[H_{{\R}_{++}}^0(G(\fil))]_{\m+\n+{\bf e_j}}=0.$\\
 \\ {\bf Case 2:} Let $\m+\n+{\bf e_j}\in\NN^s.$ Then $\m+\n+{\bf e}+{\bf e_j}\geq {\bf e}.$ Therefore from the exact sequence (\ref{tech}), we have $$[H_{{\R}_{++}}^{1}(N)]_{\m+\n+{\bf e}+{\bf e_j}}=C_{\m+\n+{\bf e}+{\bf e_j}}=0.$$ Hence from the exact sequence (\ref{ki}), $[H_{{\R}_{++}}^0(G(\fil))]_{\m+\n+{\bf e_j}}=0.$ Thus $\reg(\R(\fil))\subseteq \reg(G(\fil)).$
\eepf
\bt\label{complete reduction filter-regular}
Let $(R,\mm)$ be a Noetherian local ring of dimension $d\geq 1$ with infinite residue field $k$ and $\idl$ be $\mm$-primary ideals in $R$ with $\grade(\id)\geq 1.$ Let $\mathcal{F}=\lbrace\fil(\n)\rbrace_{\n\in{\ZZ}^s}$ be an $\I$-admissible filtration of ideals in $R.$ Let $\mathcal A_{\fil}=\lbrace a_{ij}\in I_i:j=1,\ldots,d;\ii\rbrace$ be a \crr of $\fil$ and $y_j=a_{1j}\cdots a_{sj}$ for all $j=1,\ldots,d.$ Let $\mathfrak a$ denote the ideal $(y_1^*,\ldots,y_d^*)$ in $G(\I)$ where $y_j^*=y_j+{\I}^{2{\bf e}}$ for all $j=1,\ldots,d.$ Then there exist $f_1,\ldots,f_d\in \mathfrak a_{\bf e}$ such that $f_1,\ldots,f_d$ is an $G(\fil)$-filter-regular sequence and $(y_1^*,\dots,y_d^*)=(f_1,\dots,f_d)G(\I).$
\et
\bpf Since $\fil$ is $\I$-admissible filtration of ideals in $R,$ $G(\fil)$ is finitely generated $G(\I)$-module. First we prove that if $\mathcal A_{\fil}$ is a \crr of $\fil$ then ${\mathfrak a}_{\n}=G(\I)_{\n}$ for all large $\n.$ For a complete reduction  $\mathcal A_{\fil}$ of $\fil,$ we have $(y_1,\ldots,y_d)\fil(\n)=\fil(\n+{\bf e})$ for all large $\n.$ Since $\fil$ is an $\I$-admissible filtration, by \cite[Proposition 2.4]{msv}, for each $\ii,$ there exists an integer $r_i$ such that 
 for all $\n\in\ZZ^s,$ where $n_i\geq r_i,$ $\fil(\n+{\bf {e_i}})=I_i\fil(\n).$ Thus for all large $\n,$ we get $$(y_1,\ldots,y_d)\fil(\n)=\I^{\bf e}\fil(\n).$$ Therefore by \cite[Lemma 1.5]{rees3}, we get $(y_1,\ldots,y_d)$ is a reduction of $\I^{\bf e}$ and hence  $(y_1^*,\dots,y_d^*)_{\n}=G(\I)_{\n}$ for all large $\n.$
\\Denote $G(\fil)$ by $M$ and $\displaystyle{{M}/{H^{0}_{G(\I)_{++}}(M)}}$ by $M'.$ Then $Ass(M')=Ass(M)\setminus V(G(\I)_{++}).$ Let $Ass(M')=\{P_1,\ldots,P_k\}.$ Since $\sqrt{\mathfrak a}=\sqrt{G(\I)_{++}},$ for all $j=1,\ldots,k,$ $P_j\nsupseteq{\mathfrak a}.$ Consider the $k$-vector space $\mathfrak a_{\bf e}/\mm\mathfrak a_{\bf e}.$ Then for each $j=1,\ldots,k,$ $$(P_j\cap \mathfrak a_{\bf e}+\mm \mathfrak a_{\bf e})/\mm \mathfrak a_{\bf e}\neq \mathfrak a_{\bf e}/\mm \mathfrak a_{\bf e}.$$ Since $k$ is infinite, there exists $f_1\in \mathfrak a_{\bf e}\setminus \bigcup_{j=1}^k(P_j\cap\mathfrak a_{\bf e}+\mm \mathfrak a_{\bf e}).$
\\By \cite[Proposition 4.1]{msv}, there exists $\m$ such that $[{H^{0}_{G(\I)_{++}}(M)}]_{\n}=0$ for all $\n\geq\m.$ Let $\n\geq\m$ and $x\in (0:_{M}f_1)_{\n}.$ Then $f_1x'=0$ in $M'$ where $x'$ is the image of $x$ in $M'.$ Since $f_1$ is a nonzerodivisor of $M',$ $x\in [{H^{0}_{G(\I)_{++}}(M)}]_{\n}=0.$ Hence $f_1$ is $M$-filter-regular and its image in the $k$-vector space $\mathfrak a_{\bf e}/\mm\mathfrak a_{\bf e}$ is linearly independent.
\\If $d=1$ then we are done. Suppose $d\geq 2.$ Then by above procedure we can choose $f_1\in {\mathfrak a}_{\bf e}.$ Since $k$ is infinite, by Nakayama's Lemma and Lemma \ref{very very important}, for each $i=2,\ldots,d,$ there exists an element $$\displaystyle{f_i\in {\mathfrak a}_{\bf e}\setminus\lf\lf\mm{\mathfrak a}+(\sum_{j=1}^{i-1}{f_{j}}G(\I))\rg\cup\lf\bigcup_{P\in Ass(M/(\sum\limits_{j=1}^{i-1}{f_{j}})M)\setminus V(G(\I)_{++})} P\rg \rg}.$$ Hence $f_1,\ldots,f_d$ is an $M$-filter-regular sequence and images of $f_1,\ldots,f_d$ in $\mathfrak a_{\bf e}/\mm\mathfrak a_{\bf e}$ are $k$-linearly independent. Since $\mathfrak a$ is generated by elements of degree ${\bf e},$ we get $\mathfrak a=(f_1,\ldots,f_d)G(\I).$
\eepf
\bt\label{hope}
Let $(R,\mm)$ be a Noetherian local ring of dimension $d\geq 1$ with infinite residue field and $\idl$ be $\mm$-primary ideals in $R$ with $\grade(\id)\geq 1.$ Let $\mathcal{F}=\lbrace\fil(\n)\rbrace_{\n\in{\ZZ}^s}$ be an $\I$-admissible filtration of ideals in $R$ and $\mathcal{A}={\lbrace}{x_{ij}\in{I_i}: j=1,\ldots,d;\ii}\rbrace$ be a complete reduction of $\fil.$ Then\ben{
\item for all $\m\in\reg(G(\fil)),$ we have $\m+(d-1){\bf e}\in \mathcal R_{\mathcal A}(\fil),$ 
\item Let $\rrr\in{\R_{\mathcal A_{\fil}}(\fil)}.$ Then for all  $\n\geq \rrr-(d-1){\bf e},$   $H_{G(\fil)_{++}}^d(G(\fil))_{\n}=0$ and hence $\rrr\in\reg^d(G(\fil)).$ }\een
\et
\bpf
$(1)$ By Theorem \ref{complete reduction filter-regular}, there exist $f_1\ldots,f_d\in G(\I)_{\bf e}$ such that $(y_1^*,\dots,y_d^*)=(f_1,\dots,f_d)G(\I)$ where $y_j^*=y_j+\I^{2{\bf e}}$ is the image of $y_j$ in $G(\I)_{\bf e}$ for all $j=1,\ldots,d.$
Now by Lemma \ref{re21}, for all $i\geq 0$ and $\n\in\ZZ^s,$ we have the short exact sequence \beqn \displaystyle [H^{i}_{G(\I)_{++}}(G(\fil))]_{\n}\longrightarrow \left[H^{i}_{G(\I)_{++}}\lf\frac{G(\fil)}{(f_1G(\fil))}\rg\right]_{\n}\longrightarrow [H^{i+1}_{R_{++}}(G(\fil))]_{\n-{\bf e}}.\eeqn Since $\m\in\reg(G(\fil)),$ we get $\left[H^{i}_{G(\I)_{++}}\lf{G(\fil)}/{f_1G(\fil)}\rg\right]_{\n}=0$ for all $\n\geq \m+{\bf e}.$ Now $f_2,\ldots,f_d$ is $G(\fil)/f_1G(\fil)$-filter regular sequence. Hence by Lemma \ref{need}, for all $\n\geq \m+d{\bf e},$ we get $$\left[H^{0}_{G(\I)_{++}}\lf{G(\fil)}/{(f_1,\ldots,f_d)G(\fil)}\rg\right]_{\n}=0.$$ Since ${G(\fil)}/{(y_1^*,\dots,y_d^*)G(\fil)}$ is $G(\I)_{++}$-torsion,for all $\n\geq \m+d{\bf e},$ we have $$\lf {G(\fil)}/{(y_1^*,\dots,y_d^*)G(\fil)}\rg_{\n}=0,$$ i.e. $$\fil(\n)=(y_1,\ldots,y_d)\fil(\n-{\bf e})+\fil(\n+{\bf e}).$$ Thus by Lemma \ref{all}, we get the require result.
\\$(2)$ Let $\rrr\in{\R_{\mathcal A_{\fil}}(\fil)}.$ Then for all $\n\geq \rrr+{\bf e},$ $\lf G(\fil)/(y_1^*,\dots,y_d^*)G(\fil)\rg_{\n}=0.$ Since $G(\fil)/(y_1^*,\dots,y_d^*)G(\fil)$ is $G(\fil)_{++}$-torsion, for all $\n\geq \rrr+{\bf e},$ we have $$H_{G(\fil)_{++}}^0\lf G(\fil)/(f_1,\ldots,f_d)G(\fil)\rg_{\n}=H_{G(\fil)_{++}}^0\lf G(\fil)/(y_1^*,\dots,y_d^*)G(\fil)\rg_{\n}=0.$$ Since $f_1,\dots,f_d$ is $G(\fil)$-filter-regular sequence, by Lemma \ref{re21}, for all $i=1\ldots,d,$ we get the following exact sequence \beqn\left[H_{G(\fil)_{++}}^{d-i}\lf\frac{G(\fil)}{(f_1,\dots,f_{i})G(\fil)}\rg\right]_{\n}&\longrightarrow& \left[H_{G(\fil)_{++}}^{d-(i-1)}\lf\frac{G(\fil)}{(f_1,\dots,f_{i-1})G(\fil)}\rg\right]_{\n-{\bf e}}\\&\longrightarrow &\left[H_{G(\fil)_{++}}^{d-(i-1)}\lf\frac{G(\fil)}{(f_1,\dots,f_{i-1})G(\fil)}\rg\right]_{\n}\eeqn where $f_0=0.$ Therefore we get $H_{G(\fil)_{++}}^d(G(\fil))_{\n}=0$ for all $\n\geq \rrr-(d-1){\bf e}.$
\eepf
\bp\label{first quadrant}
Let $(R,\mm)$ be a Noetherian local ring of dimension $d\geq 1$ with infinite residue field and $\idl$ be $\mm$-primary ideals in $R$ and $s\geq 2.$ Let $\mathcal{F}=\lbrace\fil(\n)\rbrace_{\n\in{\ZZ}^s}$ be an $\I$-admissible filtration of ideals in $R.$ Then $\PP(\fil)\subseteq \NN^s.$
\ep
\bpf
Suppose $\m\in\PP(\fil)$ such that $m_i<0$ for some $i\in\{1,\ldots,s\}.$ Let $j\in\{1,\ldots,s\}$ such that $j\neq i.$ \wlg $j<i.$ Consider the polynomial $$Q(X)=\po(m_1,\ldots,m_{j-1},X,m_{j+1},\ldots,m_i+1,\ldots,m_s)-\po(m_1,\ldots,m_{j-1},X,m_{j+1},\ldots,m_i,\ldots,m_s).$$ Then for any $n\geq m_j,$ by definition of admissible filtration of ideals, we get \beqn
Q(n)&=&\po(m_1,\ldots,m_{j-1},n,m_{j+1},\ldots,m_i+1,\ldots,m_s)-\po(m_1,\ldots,m_{j-1},n,m_{j+1},\ldots,m_i,\ldots,m_s)\\&=&\ho(m_1,\ldots,m_{j-1},n,m_{j+1},\ldots,m_i+1,\ldots,m_s)-\ho(m_1,\ldots,m_{j-1},n,m_{j+1},\ldots,m_i,\ldots,m_s)\\&=&0.
\eeqn  Therefore $Q(X)=0.$ Hence $e_{{\bf e_i}+(d-1){\bf e_j}}(\fil)=0$ which is a contradiction.
\eepf
We recall the following result.
\begin{theorem}\label{r3}{\rm{\cite[Theorem 4.3]{msv}}}
Let $(R,\mm)$ be a Noetherian local ring of dimension $d$ and $\idl$ be $\mm$-primary ideals of $R.$ Let $\fil=\lbrace\fil(\n)\rbrace_{\n\in{\ZZ}^s}$ be an $\I$-admissible filtration of ideals in $R.$ Then
\ben
{\item[\rm(1)] $\lm_R[H_{\R_{++}}^i(\mathcal{R}'(\fil))]_{\n}<\infty$ for all $i\geq 0$ and $\n\in\ZZ^s.$ 
\item[\rm(2)] $\po(\n)-\ho(\n)=\sum\limits_{i\geq 0}(-1)^i\lm_{R}[H_{\R_{++}}^i(\mathcal{R}'(\fil))]_{\n}$ 
for all $\n\in\ZZ^s.$}\een
\end{theorem}
\bt\label{pcj}
Let $(R,\mm)$ be a Noetherian local ring of dimension $d\geq 1$ with infinite residue field and $\idl$ be $\mm$-primary ideals in $R$ with $\grade(\id)\geq 1.$ Let $\mathcal{F}=\lbrace\fil(\n)\rbrace_{\n\in{\ZZ}^s}$ be an $\I$-admissible filtration of ideals in $R$ and $\mathcal A_{\fil}=\lbrace a_{ij}\in I_i:j=1,\ldots,d;\ii\rbrace$ be any complete reduction of $\fil.$ Let $H_{\R(\I)_{++}}^i(\R(\fil))_{\n}=0$ for all $0\leq i\leq d-1$ and $\n\geq {\bf 0}.$ Then 
\ben
{\rm
\item[(1)] for all $\rrr\in{\R_{\mathcal A_{\fil}}(\fil)}$ with $\rrr\geq (d-1){\bf e},$ we have $\rrr-(d-1){\bf e}\in \PP(\fil),$
\item[(2)] if $s\geq 2,$ for all $\n\in\PP(\fil),$ we have $\n+(d-1){\bf e}\in \R_{\mathcal A_{\fil}}(\fil).$
}\een 	In particular, for $s\geq 2,$ there exists a one-to-one correspondence 
$$f:{\displaystyle{\Large{\Large{\PP(\fil)\longleftrightarrow\lbrace{\rrr\in{\R_{\mathcal A_{\fil}}(\fil)}\mid\rrr\geq (d-1){\bf e}}\rbrace}}}}$$ defined by $f(\n)=\n+(d-1){\bf e}$ where $f^{-1}(\rrr)=\rrr-(d-1){\bf e}.$\et
\bpf
Consider the short exact sequence of $\R(\I)$-modules $$0\longrightarrow \mathcal{R}'(\fil)({\bf e})\longrightarrow \mathcal{R}'(\fil)\longrightarrow G'(\fil)\longrightarrow 0.$$ This induces the long exact sequence of $R$-modules
{\small\beqnn\label{up}\cdots\longrightarrow [H_{{\R}_{++}}^i(\mathcal{R}'(\fil))]_{\n+{\bf e}} \longrightarrow [H_{{\R}_{++}}^i(\mathcal{R}'(\fil))]_{\n} 
\longrightarrow [H_{{\R}_{++}}^i(G'(\fil))]_{\n}
 \longrightarrow [H_{{\R}_{++}}^{i+1}(\mathcal{R}'(\fil))]_{\n+{\bf e}} \longrightarrow \cdots.\eeqnn}\noindent
$(1)$ Let $\rrr\in{\R_{\mathcal A_{\fil}}(\fil)}$ with $\rrr\geq (d-1){\bf e}.$ Therefore by Theorem \ref{hope}, we get $H_{G(\fil)_{++}}^d(G(\fil))_{\n}=0$ for all $\n\geq \rrr-(d-1){\bf e}.$ Thus from the long exact sequence (\ref{up}), using Proposition \ref{result 2} and the change of ring principle, we have $[H_{\R_{++}}^d(\mathcal{R}(\fil))]_{\n}=0$ for all $\n\geq \rrr-(d-1){\bf e}.$ Therefore again using the Difference Formula (Theorem \ref{r3}), we get $\po(\n)=\ho(\n)$ for all $\n\geq \rrr-(d-1){\bf e}.$ 
\\$(2)$ Let $\m\in\PP(\fil).$ Then $\m\geq {\bf 0}$ and for all $\n\geq\m,$ by Proposition \ref{result 2}, the Difference Formula (Theorem \ref{r3}) and the change of ring principle, we have $$0=\po(\n)-\ho(\n)=(-1)^d\lm_{R}[H_{\R_{++}}^d(\mathcal{R}(\fil))]_{\n}.$$  Therefore by Proposition \ref{result 2} and the change of ring principle, we get that $H_{G(\fil)_{++}}^i(G(\fil))_{\n}=0$ for all $i\geq 0$ and $\n\geq \m.$ For all $j=1,\ldots,d,$ let $y_j=a_{1j}\dots a_{sj}.$ Then by Theorem \ref{complete reduction filter-regular}, there exist $f_1,\dots, f_d\in G(\I)_{\bf e}$ such that $f_1,\dots,f_d$ is $G(\fil)$-filter-regular sequence
and $(y_1^*,\dots,y_d^*)=(f_1,\dots,f_d)G(\I)$ where $y_j^*=y_j+\I^{2{\bf e}}$ is the image of $y_j$ in $G(\I)_{\bf e}$ for all $j=1,\ldots,d.$ Hence by Lemma \ref{need}, for all $\n\geq \m+d{\bf e},$ we have $$H_{G(\fil)_{++}}^0\lf G(\fil)/(y_1^*,\dots,y_d^*)G(\fil)\rg_{\n}=H_{G(\fil)_{++}}^0\lf G(\fil)/(f_1,\ldots,f_d)G(\fil)\rg_{\n}=0.$$ Since $G(\fil)/(y_1^*,\dots,y_d^*)G(\fil)$ is $G(\fil)_{++}$-torsion, for all $\n\geq \m+d{\bf e},$ we get $$\fil(\n)=(y_1,\ldots,y_d)\fil(\n-{\bf e})+\fil(\n+{\bf e}).$$ Therefore by Lemma \ref{all}, for all $\n\geq \m+d{\bf e},$ we have $\fil(\n)=(y_1,\ldots,y_d)\fil(\n-{\bf e}).$ Hence $\m+(d-1){\bf e}\in{\R_{\mathcal A_{\fil}}(\fil)}.$
\eepf
\begin{remark}
Note that if $\R(\fil)$ is \CM then $H_{\R(\I)_{++}}^i(\R(\fil))_{\n}=0$ for all $0\leq i\leq d-1$ and $\n\geq {\bf 0}$ \cite[Remark 3.4]{PV3}. Let $\idl$ be monomial ideals in $k[X_1,\ldots,X_d].$ Then By Hochster's Theorem \cite[Theorem 6.3.5]{BH}, $\R(\fil)$ is \CM where $\fil=\{\ov{{\I}^{\n}}\}_{\n\in\ZZ^s}.$
\end{remark}
As a consequence of Theorem \ref{pcj}, we obtain a result of Marley \cite[Theorem 2]{marley}, \cite[Corollary 3.8]{marleythesis} and Ooishi \cite[Proposition 4.10]{aooishi}. 
\bt
Let $(R,\mm)$ be a $d$-dimensional ($d\geq 1$) \CM local ring with an infinite residue filed. Let $I$ be an $\mm$-primary ideal and $\fil=\{I_n\}_{n\in\ZZ}$ be an $I$-admissible filtration such that $\depth G(\fil)_{+}\geq d-1.$ Then $r(\fil)=n(\fil)+d.$
\et
\bpf
Suppose $\depth G(\fil)_{+}= d.$ Then $H_{G(I)_{+}}^i(G(\fil))=0$ for all $0\leq i\leq d-1.$ Now $s=1$ implies $G'(\fil)=G(\fil).$ Hence by the exact sequence (\ref{up}), $H_{{\R}_{+}}^{i}(\R'(\fil))=0$ for all $0\leq i\leq d-1.$ Suppose $\depth G(\fil)_{+}\geq d-1.$ Then by \cite[Proposistion 4.2.10]{blancafort thesis}, $\depth_{{\R}_{+}}(\R'(\fil))=d.$ Hence by the exact sequence (\ref{up}), $H_{{\R}_{+}}^{i}(\R'(\fil))=0$ for all $0\leq i\leq d-1.$ Hence in both cases, $H_{{\R}_{+}}^{i}(\R'(\fil))=0$ for all $0\leq i\leq d-1.$  Now using the same procedure as in Theorem \ref{pcj}, we get the required result.\eepf
\begin{example}
\rm{
Let $R=k[|X,Y,Z|].$ Then $R$ is a $3$-dimensional regular local ring with unique maximal ideal $\mm=(X,Y,Z).$ Since $R$ is regular local ring, $\mm$ is normal ideal. Let $I=\mm$ and $J=\mm^2.$ Then $\mathcal A=\begin{pmatrix}
  X & Y & Z\\
  X^{2} & Y^{2} & Z^2
 \end{pmatrix}$ is a \crr for the filtration $\I=\lbrace I^rJ^s\rbrace_{r,s\in\ZZ}$ and $(X^3,Y^3,Z^3)I^2J^2=I^3J^3.$ Since $I,J$ are monomial ideals and integrally closed, by Hochster's theorem \cite[Theorem 6.3.5]{BH}, $\R(\I)$ is \CM and hence $H_{\R(\I)_{++}}^i(\R(\I))_{\n}=0$ for all $0\leq i\leq 3.$ Therefore by \cite[Theorem 3.5]{powerproduct}, $\pf(\n)=\hf(\n)$ for all $\n\in\NN^s.$ 
}\end{example}
In the following example we show that we cannot drop the condition on $H_{\R_{++}}^i(\R(\fil))_{\n}$ for all $0\leq i\leq d-1$ in Theorem \ref{pcj}.
\begin{example}
\rm{
Let $R=k[|X,Y|].$ Then $R$ is a 2-dimensional \CM local ring with unique maximal ideal $\mm=(X,Y).$ Let $I=\mm^2$ and $J=(X^2,Y^2).$ Since $(X^{4},Y^{4})IJ=(X^{4},Y^{4})\mm^{4}=\mm^{8}=I^2J^2,$ we have $\mathcal A=\begin{pmatrix}
  X^{2} & Y^{2}\\
  X^{2} & Y^{2}
 \end{pmatrix}$ is a \crr for the filtration $\I=\lbrace I^rJ^s\rbrace_{r,s\in\ZZ}.$ Since $\widebreve{{\I}^{\bf e_2}}=\breve{J}=(I^k{J}^{1+k}:{I^kJ^k})$ for some large $k,$ $JI=I^2$ and $\mm$ is parameter ideal, we have $$\breve{J}=(I^{2k+1}:I^{2k})=(\mm^{4k+2}:\mm^{4k})=\mm^2\neq J\mbox{ and hence by \cite[Proposition 3.5]{msv} }H_{\R_{++}}^1(\R(\I))_{(0,1)}\neq 0.$$
 Since $(X^{4},Y^{4})IJ=I^2J^2,$ we get ${\bf e}\in \R_{\mathcal A_{\I}}(\I.)$ 
 \\ For $n\geq 1,$ $$\lm\lf\frac{R}{I^n}\rg=\lm\lf\frac{R}{\mm^{2n}}\rg=\binom{2n+1}{2}=4\binom{n+1}{2}-n=P_I(n),$$ Since $J$ is parameter ideal and $JI=I^2,$ $$\lm\lf\frac{R}{J^n}\rg=4\binom{n+1}{2}=P_{J}(n),$$ $$\lm\lf\frac{R}{(IJ)^n}\rg=\lm\lf\frac{R}{\mm^{4n}}\rg=\binom{4n+1}{2}=16\binom{n+1}{2}-6n=P_{IJ}(n)$$ and
$$\lm\lf\frac{R}{I^{2n}J^{n}}\rg=\lm\lf\frac{R}{\mm^{6n}}\rg=\binom{6n+1}{2}=36\binom{n+1}{2}-15n=P_{I^2J}(n).$$ Hence $e_0(I)=e_0(J)=4.$ Now for large $n,$ $\displaystyle P_{IJ}(n)=\lm\lf\frac{R}{(IJ)^n}\rg=P_{\I}(n{\bf e})$ and $\displaystyle P_{I^2J}(n)=\lm\lf\frac{R}{I^{2n}J^{n}}\rg=P_{\I}(2n,n).$ Comparing the coefficients on both sides we get 
$$\pf(r,s)=4\binom{r+1}{2}+4\binom{s+1}{2}+4rs-r-s.$$ Then $\pf(0,1)=3\neq 4=\lm\lf{R}/{J}\rg.$ Hence $(0,0)\notin \PP(\I).$ 

}
\end{example}

\end{document}